\documentclass{amsart}

\usepackage{float}
\usepackage{aliascnt}
\usepackage{amssymb}
\usepackage{hyperref}
\usepackage[all]{xy}
\usepackage{color}
\usepackage{tikz}
\usetikzlibrary{arrows}

\hypersetup{pdftitle={Motivic Donaldson--Thomas invariants of some quantized threefolds},pdfauthor={Alberto Cazzaniga, Andrew Morrison, Brent Pym and Balázs Szendrői}}

\definecolor{sxdarkblue}{RGB}{150,180,200} 
\definecolor{sxlightblue}{RGB}{186,215,230} 
\definecolor{sxred}{RGB}{153,0,0} 
\definecolor{tocolor}{rgb}{.1,.1,.1}
\definecolor{urlcolor}{rgb}{.2,.2,.6}
\definecolor{linkcolor}{rgb}{.1,.1,.5}
\definecolor{citecolor}{rgb}{.4,.2,.1}
\hypersetup{colorlinks=true, urlcolor=urlcolor, linkcolor=linkcolor, citecolor=citecolor}

\theoremstyle{plain} 
\newtheorem{theorem}{Theorem}[section]

\newcommand{\thdef}[2]{
	\newaliascnt{#1}{theorem}  
	\newtheorem{#1}[#1]{#2}
	\aliascntresetthe{#1}  
	\newtheorem*{#1*}{#2}
	\expandafter\newcommand\expandafter{\csname #1autorefname\endcsname}{#2}
}
\thdef{corollary}{Corollary}
\thdef{lemma}{Lemma}
\thdef{conjecture}{Conjecture}
\thdef{proposition}{Proposition}

\theoremstyle{definition} 

\thdef{exercise}{Exercise}
\thdef{construction}{Construction}
\thdef{example}{Example}
\thdef{principle}{Principle}
\thdef{remarks}{Remarks}
\thdef{remark}{Remark}
\thdef{definition}{Definition}

\newcommand{\rbrac}[1]{\left(#1\right)}
\newcommand{\abrac}[1]{\left\langle#1\right\rangle}
\newcommand{\set}[2]{\left\{#1\,\middle|\,#2\right\}}

\newcommand{\M}{{\mathcal M}}
\newcommand{\OO}{{\mathcal O}}

\newcommand{\C}{\mathbb{C}}
\newcommand{\Z}{\mathbb{Z}}
\renewcommand{\P}{\mathbb{P}}
\newcommand{\N}{\mathbb{N}}
\newcommand{\A}{\mathbb{A}}

\newcommand{\LL}{\mathbb{L}}
\newcommand{\half}{\frac{1}{2}}

\newcommand{\gln}[1]{\mathfrak{gl}_{#1}}
\newcommand{\h}{\mathfrak{h}}

\def\cC{\mathbb{C}}

\def\cL{\mathbb{L}}
\def\cN{\mathbb{N}}
\def\cP{\mathbb{P}}

\def\cZ{\mathbb{Z}}

\def\al{{\alpha}}

\def\be{{\beta}}
\def\pt{\rm pt}
\def\rinv{{\rm inv}}
\newcommand{\rL}[1][1]{\cL^{\frac{#1}{2}}}
\newcommand{\irL}[1][1]{\cL^{-\frac{#1}{2}}}
\newcommand{\stddenom}{\rL-\irL}

\def\Spec{\operatorname{Spec}}

\def\Hom{\operatorname{Hom}}
\def\End{\operatorname{End}}
\def\GL{\operatorname{GL}}
\def\SL{\operatorname{SL}}

\def\Exp{\operatorname{Exp}}

\def\crit{\operatorname{crit}}
\def\udim{\operatorname{\underline\dim}}
\def\vir{\mathrm{vir}}

\def\dd{{\partial}}
\def\ms{\backslash} 

\def\mto{\mapsto}

\def\inv{^{-1}}

\def\oh{{\frac12}}

\def\denom{\cL^\oh-\cL^{-\oh}}
\def\gM{\mathfrak{M}}

\def\sF{\mathcal{F}}
\def\GG{G}

\def\pt{\mathrm{pt}}
\def\con{\mathrm{con}}
\def\DT{\mathrm{DT}}
\begin{document}

\title[Motivic Donaldson--Thomas invariants of quantized threefolds]{Motivic Donaldson--Thomas invariants \\ of some quantized threefolds}
\author{Alberto Cazzaniga, Andrew Morrison, Brent Pym and Bal\'azs
  Szendr\H oi}

\maketitle

\vspace{-0.1in}
\begin{center} 
{\it In memory of Kentaro Nagao}
\end{center}

\begin{abstract}
This paper is motivated by the question of how motivic Donaldson--Thomas invariants behave in families. We compute the invariants for some simple families of noncommutative Calabi--Yau threefolds, defined by quivers with homogeneous potentials.  These families give deformation quantizations of affine three-space, the resolved conifold, and the resolution of the transversal $A_n$-singularity. It turns out that their invariants are generically constant, but jump at special values of the deformation parameter, such as roots of unity. The corresponding generating series are written in closed form, as plethystic exponentials of simple rational functions. While our results are limited by the standard dimensional reduction techniques that we employ, they nevertheless allow us to conjecture formulae for more interesting cases, such as the elliptic Sklyanin algebras.
\end{abstract}

\medskip

\thispagestyle{empty}

\section{Introduction}

Donaldson--Thomas (DT) invariants were introduced by Donaldson and Thomas in~\cite{DT, T} to give 
a numerical count of sheaves 
on three-dimensional projective Calabi--Yau varieties. One of the fundamental results of~\cite{T} was the
statement that these integer quantities are indeed invariants, in the sense that they are unchanged 
when the 
underlying Calabi--Yau threefold moves in a connected projective family. It was later realized that DT-like 
invariants can be defined by counting objects in more general 3-Calabi--Yau categories, such as categories
defined by a quiver with potential~\cite{Sze}.
In a different direction, building on work of Behrend~\cite{Be}, DT-like
invariants taking values in more general rings and not just in $\Z$, such as rings of (naive) 
motives~\cite{BBS, KS}, were defined. 

In this paper, we are interested in how motivic DT invariants behave in families. In the projective
case, this problem seems difficult to study, since it is hard to compute motivic DT invariants of 
projective Calabi--Yau varieties in all but a handful of cases. 
Here, we instead study deformation properties of motivic DT invariants for some families of \emph{noncommutative} Calabi--Yau threefolds, defined by quivers with homogeneous potentials. The motivic invariants we look at are attached to moduli spaces of finite-dimensional representations of the Jacobian algebra associated to the quiver with potential; this should be seen as the noncommutative analogue of studying moduli of finite sets of points on commutative threefolds. Moduli spaces of homogeneous potentials were studied recently in~\cite{OU} in a specific example, where the question of the behaviour of motivic DT invariants was also raised. We can regard a family of {\em graded} deformations of a homogeneous potential as an analogue of a projective family in the local noncommutative situation. 

The results we are reporting on here are rather limited.  In particular, we do not introduce any new techniques for computing DT invariants in this paper; rather we use the 
dimensional reduction technique already used in~\cite{BBS}, and systematized in~\cite{Mor},
to see what we can learn about the structure of the invariants and their deformation properties. 
Thus, the cases in which we have been able to compute the generating series for the motivic DT invariants 
are the cases in which the potential is linear with respect to one of the generators of the algebra.

We explore only the simplest deformations
of potentials for some well-studied quivers of geometric origin: the three-loop quiver
underlying the Hilbert scheme of points on threefolds; the conifold quiver; and, generalizing the first,
the cyclic quiver with loops. In the case of undeformed potentials, corresponding to commutative Calabi--Yau threefolds, 
the motivic invariants in these examples
were computed in \cite{BBS}, \cite{MMNS} and \cite{BM, Mor}, respectively. We study some simple perturbations  of these potentials, corresponding to certain deformation quantizations of the commutative threefolds, and compute or conjecture the corresponding motivic DT invariants.    
One reason for focusing on the three-loop quiver, in particular, is that such homogeneous deformations of the potential correspond to marginal deformations of $N=4$ super Yang--Mills theory~\cite{BJL}; our formulae therefore correspond to refined BPS counts for these deformed theories~\cite{DG}.

The motivic DT invariants for our families are certainly not deformation-invari\-ant. However, we find that they behave in rather
well-controlled ways, having constant generic value, but jumping at special values of parameters (such as roots of unity) for which the quantized algebras become finite modules over their centres.  
At least for our simple quivers, the generating series of motivic DT invariants have a surprisingly simple form:
they are motivic ``Exponentials'' of rather simple rational functions of the vertex variables. The fact that 
generating functions of DT invariants are Exponentials is now well known and underlies rationality
(also known as BPS) conjectures and theorems~\cite{JS,KS}. 
However, the fact that inside the Exponential we often have simple
rational functions seems not to have been observed before in this generality.

The rational functions, including their coefficients, appear to be governed by {\em simple} 
finite-dimensional representations of the corresponding Jacobi algebra. In the 
case of the three-loop quiver, based on our results (Theorems~\ref{thm_quantumC3}-\ref{thm_jordan}) we are able to articulate a somewhat more precise conjectural 
formula~\eqref{conj_Q1}, and use it to predict the answer in some cases 
which we cannot access via dimensional reduction---the homogenized Weyl algebra and the elliptic Sklyanin algebras (Conjectures~~\ref{conj_weyl}-\ref{conj_Skly}). 
In multi-vertex cases, our results (Theorems~\ref{thm_qconifold}-\ref{thm:cyclicAn}) do not lead to any precise conjecture. Nevertheless, the coefficients in the rational functions still have intriguing geometric interpretations: they seem to correspond to the degeneracy loci of the Poisson structures that are quantized in order to produce the given families of noncommutative algebras.  We thus uncover an intriguing landscape, which we leave for further study.

The paper is organized as follows.  In \autoref{sec_prelim}, we give a brief review of the necessary preliminaries about quiver representations; the ring of motivic classes; and the definition of motivic DT invariants and their generating series.  In \autoref{sec_results}, we summarize our computations and conjectures for the generating series in the examples of interest, and we discuss their geometric interpretations.  We finish in \autoref{sec_proofs} with the proofs.

\smallskip

\noindent{\bf Acknowledgements.} A.C.~and B.Sz.~were supported by EPSRC Programme Grant EP/I033343/1 during the preparation of this paper, and B.P.~was supported by EPSRC Grant EP/K033654/1.  The authors would like to acknowledge helpful conversations and correspondence with Ivan Cheltsov, Ben Davison, Tommaso de Fernex, Kevin De Laet, Dominic Joyce, Toby Stafford and Chelsea Walton.
\section{Preliminaries}
\label{sec_prelim}

\subsection{Quivers and their representations}

Let $Q$ be a finite quiver, with vertex set $V(Q)$ and arrow set $E(Q)$.
For an arrow $a\in E(Q)$, denote by $s(a)\in V(Q)$, respectively
$t(a)\in V(Q)$, the vertex at which $a$ starts, respectively ends.
The Euler-Ringel form $\chi$ on $\cZ^{V(Q)}$ is 
$$\chi(\al,\be)=\sum_{i\in V(Q)}\al_i\be_i-\sum_{a\in E(Q)}\al_{s(a)}\be_{t(a)},\qquad \al,\be\in\cZ^{V(Q)}.$$
Given a $Q$-representation $M$, its dimension vector 
$\udim M\in\cN^{V(Q)}$ is $\udim M=(\dim M_i)_{i\in V(Q)}$.
Let $\al\in\cN^{V(Q)}$ be a dimension vector and let $V_i=\cC^{\al_i}$, $i\in V(Q)$. Let
$$R(Q,\al)=\bigoplus_{a\in {E(Q)}}\Hom(V_{s(a)},V_{t(a)})$$
and 
\[\GG_\al=\prod_{i\in V(Q)}\GL(V_i).\]
Then $\GG_\al$ naturally acts on $R(Q,\al)$, and the quotient stack
$$\gM(Q,\al)=[R(Q,\al)/\GG_\al]$$
gives the moduli stack  of representations of $Q$ with dimension vector $\al$.

Let $W$ be a potential on $Q$, a finite linear combination of cyclic paths in $Q$. Denote by $J_{Q,W}$ 
the Jacobian algebra, the quotient of the path algebra $\C Q$ by the two-sided ideal generated by formal partial
derivatives of the potential $W$. Let
$$f_\al:R(Q,\al)\to\cC$$ be the $\GG_\al$-invariant function defined by taking the trace of the map associated 
to the potential $W$. A point in the critical locus 
$\crit(f_\al)$ corresponds to a $J_{Q,W}$-module. The quotient stack 
$$\gM(J_{Q,W}, \al)=\bigl[\crit(f_\al)/\GG_\al\bigr]$$
gives the moduli stack of $J_{Q,W}$-modules with dimension vector $\al$.

\subsection{The ring of motivic classes}

Let $K^{\hat\mu}({\rm Var}_{\mathbb C})$ be the ring of isomorphism 
classes of reduced varieties over ${\mathbb C}$, equipped with a good action
of a finite group of roots of unity, respecting the scissor relation
and the relation $[\A^n, \mu_k] = [\A^n]$ for a linear representation 
of the group $\mu_k$ on affine space $\A^n$, as in~\cite{L}.  Here $\mu_k$ denotes the group of $k$th roots of unity.

Let $\cL=[\A^1]\in K^{\hat\mu}({\rm Var}_{\mathbb C})$ be the class of the affine line with trivial action.  Then it is known that $\cL$ admits a square root $\cL^{\half} \in K^{\hat\mu}({\rm Var}_{\mathbb C})$.  We work in the motivic ring
\[
\M = \left(K^{\hat\mu}({\rm Var}_{\mathbb C})/\mathrm{Ann}(\cL)\right)[\cL^{-\half}, (1-\cL^n)^{-1}\colon n\geq 1],
\]
where $\mathrm{Ann}(\cL)$ denotes the annihilator~\cite{Bo} of $\cL$ in $K^{\hat\mu}({\rm Var}_{\mathbb C})$. 
We have the Euler characteristic specialization defined on
classes of varieties by $[X]\mapsto \chi(X)$ (with or without compact support, which agree),
and $\cL^\oh\mto -1$; this is of course well-defined only
on a subset of elements of $\M$ which we call {\em motives without denominators}. 
 
For a regular function $f\colon X \to
\cC$ on a smooth variety $X$, Denef and Loeser \cite{DL,L} define the motivic nearby 
cycle $[\psi_f]\in \M$ and the motivic vanishing cycle
$ [\varphi_f]=[\psi_f]-[f\inv(0)]\in \M$ of~$f$ (at the value $0\in \cC$, the only case
that will be relevant to us since all our functions $f$ will be homogeneous). 
For $f=0$, we have $[\varphi_0]=-[X]$.
Given a global critical locus
$Z=\{df=0\}\subset X$ for $f\colon X\to\C$ on a smooth complex variety
$X$, define the virtual motive of $Z$ by 
\[
[Z]_\vir = -(-\cL^\oh)^{-\dim X}[\varphi_f]\in \M.
\]
Thus for a smooth variety $X$ with $f=0$, we have
\[
[X]_\vir=(-\cL^\oh)^{-\dim X}\cdot [X].
\]

The ring $\M$ is known to be a so-called pre-$\lambda$-ring, with operations $\sigma_n:\M \to \M$ for $n\geq 0$
satisfying certain natural compatibilities~\cite{GSLMH, DM}. For classes $[X]\in\M$ represented by 
quasiprojective varieties $X$, we have $\sigma_n([X])=[\mathrm{Sym}^n(X)]$, and thus $\sigma_n(\cL)=\cL^n$; we also 
have $\sigma_n(-\cL^\half) = (-\cL^\half)^n$. There is an induced 
pre-$\lambda$-ring structure on the power series ring $\M[[t_1, \ldots, t_k]]$ defined by
$\sigma_n(rt_j^i)= \sigma_n(r)t_j^{ni}$. Denoting finally by 
$\M[[t_1, \ldots, t_k]]_+\subset \M[[t_1, \ldots, t_k]]$ the ideal generated by $t_1, \ldots, t_k$, we define the plethystic exponential
\[\Exp\colon \M[[t_1, \ldots, t_k]]_+\to 1+\M[[t_1, \ldots, t_k]]_+\]
by the formula
\[ \Exp(r) = \sum_{n\geq 0} \sigma_n(r).
\]
See~\cite[Section 2.5]{BBS} as well as~\cite[Section 2]{DM2} for a more leisurely introduction 
to aspects of this formalism and some more explicit formulae. 

\subsection{Statement of the problem}

Given a quiver with potential $(Q,W)$, we define motivic Donaldson-Thomas invariants
\[ [\gM(J_{Q,W},\al)]_\vir=\frac{[\crit(f_\al)]_\vir}{[\GG_\al]_\vir},\]
where $[\GG_\al]_\vir$ refers to the virtual motive of the pair $(\GG_\al,0)$.  
We package these invariants into a generating series by introducing a set $t=(t_i\colon i\in V(Q))$ of auxiliary variables, and setting
\[U_{Q,W}(t) =\sum_{\al\in\cN^{V(Q)}}[\gM(J_{Q,W},\al)]_\vir \cdot t^\al
,\]
where we use multi-index notation for monomials $t^\al$.
Our aim is to compute the series $U_{Q,W}(t)$ in closed form for some interesting classes of pairs $(Q,W)$. We 
will particularly be interested in how $U_{Q,W}(t)$ changes for a fixed $Q$ under homogeneous
deformations of the potential $W$. 

The series $U_{Q,W}$ is called the universal DT series in~\cite{MMNS}. Generating series of framed 
invariants are related to $U_{Q,W}$ by wall crossing~\cite{JS, KS, Moz, N}.

\section{Results and interpretations}
\label{sec_results}

\subsection{Deformations of  affine three-space}

Let $Q_1$ be the quiver corresponding to affine three-space~\cite{BBS}, with $V(Q_1)$ containing
a single vertex and $E(Q_1)=\{x,y,z\}$ consisting of three loops based at the vertex as shown in \autoref{fig_Q1}.
\vspace{-2em}
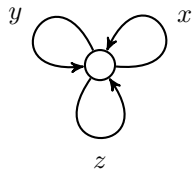
\begin{figure}[H]
\begin{center}
\begin{tikzpicture}[>=stealth',auto,node distance=1cm,
  thick,main node/.style={circle,draw,inner sep=0pt,minimum size=0.4cm}]
\node[main node] (1) {};
\draw [->] (1) .. controls (-5:1.5) and (65:1.5) .. (1);
\draw [->] (1) .. controls (115:1.5) and (185:1.5) .. (1);
\draw [->] (1) .. controls (235:1.5) and (305:1.5) .. (1);
\draw (30:1.3) node {$x$};
\draw (150:1.3) node {$y$};
\draw (270:1.3) node {$z$};
\end{tikzpicture}
\caption{The quiver $Q_1$.}\label{fig_Q1}
\end{center}
\end{figure}
\vspace{-2em}
When the potential on $Q_1$ is  $W=xyz-xzy$, the Jacobian algebra $J_{Q_1,W}$ is simply the polynomial ring $\C[x,y,z]$, corresponding to the simplest commutative Calabi--Yau threefold: affine three space $\C^3$.  This perspective was used in~\cite{BBS} to compute the motivic DT invariants of Hilbert schemes of threefolds.

We will consider homogeneous cubic deformations of the potential $W$, resulting in flat deformations of the polynomial ring as a graded Calabi--Yau algebra.  These algebras correspond to quantizations (i.e.~noncommutative deformations) of~$\C^3$ and its projectivization~$\cP^2$.  Such deformations have received substantial attention in the noncommutative geometry literature, starting with~\cite{ATVdB,AS}.  In particular, considerable work has been done on their representation theory~\cite{ATVdB2,DeL,DeLLeB,W}.

\subsubsection{Quantum affine three-space}

Start with the potential \[W_q=xyz-q xzy\] on $Q_1$, where $q\in\C^*$ is a constant.   The corresponding Jacobian algebra 
$J_{Q_1, W_q}$ is the coordinate ring of quantum affine three-space
\[J_{Q_1,W_q} = \C\abrac{x, y, z} / \rbrac{ [x,y]_q, [y,z]_q, [z, x]_q }
\]
with $[a,b]_q=ab-qba$.  The requirement that $q$ be nonzero is important, as it ensures that the algebra is Calabi--Yau.

Since the quiver has only one vertex, a dimension vector is just a single number, indicating the dimension of the representation, so that the universal series is a function of a single variable $t$.  Corresponding to the fact that $W_q$ is linear in the generator $z$, the algebra $J_{Q_1,W_q}$ has an extra $\C^*$-symmetry, given by rescaling $z$.  In \autoref{sec_qC3_proof}, we exploit this symmetry to prove the following

\begin{theorem} \label{thm_quantumC3}
If $q\in\C^*$ is a primitive $r$th root of unity, then 
\[U_{Q_1,W_q}(t) = \Exp\left( \frac{2\mathbb{L}-1}{\mathbb{L}-1}\frac{t}{1-t} + (\mathbb{L}-1)\frac{t^r}{1-t^r} \right). \]
Otherwise, 
\[ U_{Q_1,W_q}(t) =  \Exp\left( \frac{2\mathbb{L}-1}{\mathbb{L}-1}\frac{t}{1-t} \right).\]
\end{theorem}

\subsubsection{The Jordan deformation}

Up to isomorphism, the only other deformation of $W$ that is linear in one of the generators is given by
\[
W_J = xyz-xzy - zy^2.
\] 
One of the relations of the Jacobian algebra $J_{Q_1, W_J}$ is $[x,y]=y^2$, which is precisely the relation between 
the generators $x,y$ of the non-commutative affine plane commonly known as the Jordan plane.  In \autoref{sec_Jordan_proof}, we prove:

\begin{theorem} For the Jordan deformation, we have
\[
U_{Q_1,W_J}(t) = \Exp\left(\frac{\cL}{\cL-1}\frac{t}{1-t}\right).
\]
\label{thm_jordan}\end{theorem}

\subsubsection{The homogenized Weyl deformation} 

Consider now the potential 
\[
W_{hW} = xyz-xzy -\frac{1}{3} z^3
\]
on the three-loop quiver. One of the relations of the Jacobian algebra $J_{Q_1, W_W}$
is $[x,y]=z^2$, the homogenization of the Weyl algebra relation $[x,y]=1$.  In \autoref{sect_interpret} below, we will explain the following

\begin{conjecture} We have
\[
U_{Q_1,W_{hW}}(t) = \Exp\rbrac{ \frac{\cL(1-[\mu_3])}{\cL-1}\frac{t}{1-t}},
\]
where by slight abuse of notation we denote by $[\mu_3]$ the equivariant 
motivic class of $\{z^3=1\}\subset\C$ carrying the canonical action of $\mu_3$.
\label{conj_weyl}
\end{conjecture}

\subsubsection{Sklyanin deformations}

Consider finally the family of Sklyanin deformations
\[ W_{a,b,c} = a\, xyz + b\, xzy + \frac{c}{3}(x^3 + y^3 + z^3),
\]
for $[a:b:c]\in\P^2$.  We assume that $abc\neq 0$ and $(3abc)^3 \neq (a^3+b^3+c^3)^3$.  The relations in the 
Jacobian algebra $J_{Q_1,W_{a,b,c}}$ are given by
\[
a xy + b yx + c  z^2 = a yz + b zy + c x^2 = a zx + b  xz + c  y^2 = 0.
\]
As explained in \cite[p.~38]{ATVdB}, to each such algebra is associated a pair of smooth cubic curves in $\cP^2$. 
 The first is the \emph{point scheme} $E_{pt}$, which parametrizes isomorphism classes of graded modules with Hilbert series $(1-t)^{-1}$; it is given explicitly by 
\[
E_\pt=\left\{(a^3+b^3+c^3)XYZ-abc(X^3+Y^3+Z^3) = 0\right\} \subset \cP^2,
\]
using homogeneous coordinates $[X:Y:Z]$ on $\cP^2$. The functor which shifts the grading of a module induces a translation $\sigma : E_\pt \to E_\pt$ in the group law of the cubic curve.
The second cubic curve, which typically has a different $j$-invariant, is the curve $E_\DT$ defined by the vanishing of the potential $f : \C^3\to \C$ for one-dimensional representations:
\[
E_\DT =\left\{ (a+b)XYZ + \tfrac{c}{3}(X^3+Y^3+Z^3) = 0 \right\} \subset \P^2.
\]
Two Sklyanin algebras determine the same pair $(E_\pt,E_\DT)$ if and only if they are either isomorphic or opposite as graded algebras.  In \autoref{sect_interpret}, we explain the following conjecture. 

\begin{conjecture}\label{conj_Skly} Let $J_{Q_1,W_{a,b,c}}$ be a Sklyanin algebra as above, let $(E_\pt,E_\DT)$ be the associated elliptic curves, and let $\sigma : E_\pt \to E_\pt$ be the induced automorphism. Let 
$S_\DT$ be the affine cubic surface
\[ S_\DT= \left\{ (a+b)xyz + \tfrac{c}{3}(x^3+y^3+z^3) = 1 \right\} \subset \A^3,\]
the universal cover of $\P^2\setminus E_\DT$, carrying the canonical action of $\mu_3 \cong \pi_1(\P^2\setminus E_\DT)$. Define the virtual motive
\[ M_1 = \irL[3] \rbrac{[S_\DT,\mu_3] - [E_\DT](\cL-1) - 1}.
\]
Then we have the following conjectural formulae for the universal series:
\begin{enumerate}
\item If $|\sigma| = \infty$, then
\[
U_{Q_1,W_{a,b,c}}(t) = \Exp\rbrac{-\frac{M_1}{\cL^{\frac{1}{2}}-\mathbb{L}^{-\frac{1}{2}}} \frac{t}{1-t}}.
\]
\item If $|\sigma|$ is finite but not a multiple of three, then
\[
U_{Q_1,W_{a,b,c}}(t) = \Exp\rbrac{ -\frac{M_1}{\stddenom} \frac{t}{1-t}
- \frac{M_\sigma}{\stddenom} \frac{t^{|\sigma|}}{1-t^{|\sigma|}}},
\]
where the virtual motive $M_\sigma = \rL([\P^2] - [E_\pt/\sigma])$ involves the elliptic curve $E_\pt/\sigma$ 
isogeneous to $E_\pt$. 
\item If $|\sigma|$ is a multiple of three, then 
\[
U_{Q_1,W_{a,b,c}}(t) = \Exp \rbrac{ -\frac{M_1}{\stddenom} \frac{t}{1-t} - \sum_{r} \frac{M_r}{\stddenom}{\frac{t^r}{1-t^r} } }.
\]
where the summation index $r$ ranges over a subset of  $\{|\sigma|/3 ,\ldots,|\sigma|\}$ and $M_r$ are virtual motives without denominators.
\end{enumerate}
\end{conjecture}

\subsection{A deformation of the conifold algebra}

Let $Q_\con$ be the conifold quiver, with $V(Q_\con)=\{v_1, v_2\}$, and four arrows 
$E(Q_\con)=\{a_{1}, a_2,b_1,b_2 \}$ with head and tail as indicated in \autoref{fig_Q2}.

\begin{figure}[h]
\begin{center}
\begin{tikzpicture}[>=stealth',auto,node distance=2cm,
  thick,main node/.style={circle,draw,inner sep=2pt,minimum size=0.2cm}]
\node[main node] (1) {$v_1$};
\node[main node] (2) [right of=1] {$v_2$};
\draw [->] (1) .. controls (0.2,1) and (1.8,1) .. (2);
\draw [->] (1) .. controls (0.8,0.3) and (1.2,0.3) .. (2);
\draw [->] (2) .. controls (1.8,-1) and (0.2,-1) .. (1);
\draw [->] (2) .. controls (1.2,-0.3) and (0.8,-0.3) .. (1);
\draw (1,0.45) node {$a_1$};
\draw (1,-0.5) node {$b_1$};
\draw (1,1.05) node {$a_2$};
\draw (1,-1.08) node {$b_2$};
\end{tikzpicture}
\caption{The quiver $Q_\con$.}\label{fig_Q2}
\end{center}
\end{figure}
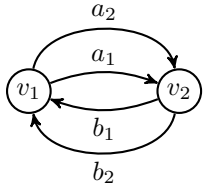
  
The standard quartic potential of the conifold quiver
\[
W=a_1b_1a_2b_2 - a_1b_2a_2b_1
\]
gives a Calabi--Yau algebra $J_{Q_\con,W}$ whose centre is given by
\[Z = \C[x,y,z,t]/(xt-yz).
\]
In this way, $J_{Q_\con,W}$ is a noncommutative crepant resolution of the conifold singularity $\Spec Z$ in the sense of~\cite{VdB}.  In particular, $J_{Q_\con,W}$ is derived equivalent to a standard (commutative) crepant resolution of $Z$.

We may therefore think of the one-parameter deformation
\[
W_q=a_1b_1a_2b_2 - q a_1b_2a_2b_1
\]
with $q\in\C^*$ as a quantization of the resolved conifold.  The condition $q\neq 0$ once again corresponding to the Calabi--Yau property for the Jacobian algebra. 

\begin{theorem} If $q\in\C^*$ is not a root of unity, then
\[ U_{Q_\con, W_q}(t_0, t_1)=\Exp\left(\frac{3\cL^{\frac{1}{2}}-\cL^{-\frac{1}{2}}}{\cL^{\frac{1}{2}}-\cL^{-\frac{1}{2}}}\frac{t_{0}t_{1}}{1-t_{0}t_{1}}-\frac{1}{\cL^{\frac{1}{2}}-\cL^{-\frac{1}{2}}}\frac{t_{0}+t_{1}}{1-t_{0}t_{1}}\right).
\]
If $q$ is a primitive $r$-th root of unity, then the above expression
is multiplied by a further factor
\[ \Exp\left((\cL-1)\frac{t_{0}^rt_{1}^{r}}{1-t_{0}^rt_{1}^{r}}\right).
\]
\label{thm_qconifold} 
\end{theorem}

More general deformations of the conifold potential are studied in~\cite{C,OU}.  We leave the investigation of motivic DT invariants for these deformations for future work. 

\subsection{A deformation of the cyclic quiver}

Finally let $n\geq 1$ and consider the quiver $Q_{n+1}$ with $V(Q_{n+1})=\{v_0, \ldots,v_n\}$ and 
arrows depicted as in \autoref{fig_Q3}. 

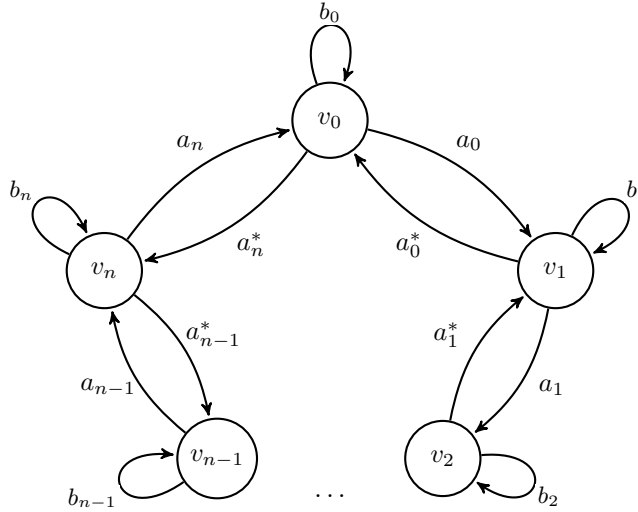
\begin{figure}[h]
 \begin{center}
\begin{tikzpicture}[->,>=stealth',shorten >=1pt,auto,node distance=3cm,
  thick,main node/.style={circle,draw,minimum size=10mm}]

  \node[main node] (0) at (0,0) {$v_{0}$};
  \node[main node] (1) at (3,-2) {$v_{1}$};
  \node[main node] (2) at (1.5,-4.5) {$v_{2}$};
  \node[main node] (n-1) at (-1.5,-4.5) {$v_{n-1}$};
  \node[main node] (n) at (-3,-2) {$v_{n}$};
  \node (a0) at (1.85,-0.3){$a_{0}$};
  \node (a0*) at (1.05,-1.65){$a_{0}^{\ast}$};
  \node (an) at (-1.85,-0.3){$a_{n}$};
  \node (an*) at (-1.05,-1.65){$a_{n}^{\ast}$};
  \node (an-1*) at (-1.55,-2.85){$a_{n-1}^{\ast}$};
  \node (a1*) at (+1.55,-2.85){$a_{1}^{\ast}$};
  \node (an-1) at (-2.95,-3.55){$a_{n-1}$};
  \node (a1) at (2.95,-3.55){$a_{1}$};
  \node (dot) at (0,-5){$\ldots$};
  \path[every node/.style={font=\sffamily\small,
  		fill=white,inner sep=1pt}]
    (0) edge [in=70, out=110, loop] node {$b_{0}$}(0)
        edge [bend left=20] node[right=1mm] {} (1)        
        edge [bend left=20] node[right=1mm] {} (n)
    (1) edge [in=115-90, out=155-90,loop] node {$b_{1}$}(1)
        edge [bend left=20] node[right=1mm] {} (2)
        edge [bend left=20] node[right=1mm] {} (0)
    (2) edge [in=115-150, out=155-150, loop] node {$b_{2}$}(2)
        edge [bend left=20] node[right=1mm] {} (1)

    (n-1) edge [in=115+60, out=155+60, loop] node {$b_{n-1}$}(n-1)
        edge [bend left=20] node[right=1mm] {} (n)
    
    (n) edge [in=115, out=155, loop] node {$b_{n}$}(n)
        edge [bend left=20] node[right=1mm] {} (0)
        edge [bend left=20] node[right=1mm] {} (n-1);
      
    \end{tikzpicture}
    \caption{The $A_{n}$-quiver $Q_{n+1}$.}\label{fig_Q3}
    \end{center}
    \end{figure}

As is well known, $Q_{n+1}$ is the McKay quiver for the embedding $\mu_{n+1}<\SL(3,\mathbb C)$ 
with weights $(1,-1,0)$. We are interested in deformations of the potential
\begin{equation}\label{AnSup}
W_1=\sum_{i=0}^{n}\left(b_{i+1}a_{i}^{\ast}a_{i}-b_{i}a_{i}a_{i}^{\ast}\right),
\end{equation}
where the index labels are to be understood modulo $n+1$.  Once again, the Jacobi algebra  $J_{Q_{n+1},W_1}$ is a non-commutative crepant resolution of its center
\[ Z \cong {\mathbb C}[x,y,z,t]/(xy-z^{n+1}).\]
The threefold $\Spec Z$ has transverse $A_n$ singularities along a line. 
 
We concentrate on the family of homogeneous  deformations of the potential $W_1$ given by 
\begin{equation}\label{def-McKay-pot}
 W_{\underline{q}}=\sum_{i=0}^{n-1}\left(q_{i}'b_{i+1}a_{i}^{\ast}a_{i}-q_{i}b_{i}a_{i}a_{i}^{\ast}\right),
\end{equation} 
where we assume $\prod_{j}q_{j}'q_{j}\neq 0$.  It is straightforward to verify that any such potential is equivalent (by rescaling variables) to a member of the one-parameter family
\[
 W_{q}=\sum_{i=0}^{n-1}\left(b_{i+1}a_{i}^{\ast}a_{i}-qb_{i}a_{i}a_{i}^{\ast}\right),
\]
parametrized by $q \in \C^*$.  The condition $q \ne 0$ is still equivalent to the Jacobian algebra being Calabi--Yau.

Let $t=(t_0, \ldots, t_n)$ be the variables in the universal DT series; we continue to use multi-index notation for monomials in these variables. Let $\delta_{i}$ be dimension vector which takes the value $1$ on $i$th vertex and $0$ otherwise. Define
\[
 \Delta=\{\delta_{i}+\ldots+\delta_{i+k}\colon i\in\{0, \ldots, n\}, k\in\{0, \ldots, n-1\}\},
\]
where we read the indices modulo $n+1$.  Thus, for example, the vector $\delta_{n}+\delta_{0}$ lies in $\Delta$. Finally, we let $\delta=\sum_{i=0}^n \delta_i$. 

\begin{theorem}\label{thm:cyclicAn}  If $q\in\C^*$ is not a root of unity, then
\[ U_{Q_{n+1}, W_q}(t)=
 \Exp\left(\frac{(n+1)\cL^{\frac{1}{2}}-\cL^{-\frac{1}{2}}}{\cL^{\frac{1}{2}}-\cL^{-\frac{1}{2}}}\frac{t^{\delta}}{1-t^{\delta}}+\sum_{\alpha\in\Delta}\frac{\cL^\half}{\cL^{\frac{1}{2}}-\cL^{-\frac{1}{2}}}\frac{t^{\alpha}}{1-t^{\delta}}\right).
\]
If $q$ is a primitive $r$-th root of unity, then the above expression
is multiplied by a further factor
\[ \Exp\left((\cL-1)\frac{t^{r\delta}}{1-t^{r\delta}}\right).
\]
\end{theorem}

\subsection{Interpretations}
\label{sect_interpret}

In this section, we analyze the results stated above. The pattern that will emerge is that, at least in our examples, the universal series has the shape of an Exponentiated rational function in the vertex variables with coefficient motives closely related to motives of spaces of simple representations of the corresponding algebra.

\subsubsection{Interpreting the results on the three-loop quiver}

Let us start by interpreting our results on the quiver $Q_1$ that leads to variants of affine three-space. 
For the potential $W_q=xyz-qxzy$ at $q=1$, we recover the result of~\cite{BBS} for the commutative case:
\[
U_{Q_1,W_1}(t) = \Exp\left(-\frac{-\cL^{3/2}}{\cL^\half-\cL^{-\half}}\frac{t}{1-t}\right).
\]
Here the motive appearing in the numerator, $-\cL^{3/2}$, is the virtual motive of 
affine three-space $\C^3$, the moduli space of simple one-dimensional modules of the algebra $U_{Q_1,W_1}$.
These are clearly the only simple modules. 

For generic quantum three-space, with parameter $q\in\C^*$ not a root of unity, 
by \autoref{thm_quantumC3} above we have
\begin{equation}
U_{Q_1,W_q}(t) = \Exp\left(-\frac{-2\cL^{\oh}+\cL^{-\oh}}{\denom}\frac{t}{1-t}\right).\label{eq_quantumC3}
\end{equation}
The only simple modules are still the one-dimensional modules by
\autoref{lem_quantumC3} below; these are parametrized by the subscheme $\{df=0\}\subset\C^3$
where $f\colon \C^3\to\C$ is given by $f(x,y,z)=xyz$. The motivic vanishing
cycle of $f$ is $2\cL^2-\cL$, so its critical locus has
virtual motive $-2\cL^{\oh}+\cL^{-\oh}$, which is the expression appearing in the 
numerator of~\eqref{eq_quantumC3}.

Let us now turn to the Jordan deformation, with 
\[J_{Q_1,W_J} = \C\langle x, y, z\rangle / ([x,y]-y^2, [y,z], [z,x]-2yz).
\]
Here we have, by \autoref{thm_jordan}, 
\begin{equation}
U_{Q_1,W_J}(t)  = \Exp\left(-\frac{-\cL^{\half}}{\cL^\half-\cL^{-\half}}\frac{t}{1-t}\right).
\label{eq_jordan}
\end{equation}
On the other hand, we still only have one-dimensional simple representations. 
\begin{lemma}\label{lem_jordan}
Every finite-dimensional simple module for $J_{Q_1,W_J}$ has dimension one.
\end{lemma}
\begin{proof}
Let $V$ be a finite-dimensional simple module.  We have the relation
\[
y^{n+2} = y^n[x,y] = [y^nx,y],
\]
which implies that every positive power of $y^2$ is a commutator, and therefore acts with trace zero on $V$.  Hence $y^2$ acts nilpotently on $V$, and so the action of $y$ on $V$ must have a nontrivial kernel $K \subset V$.  But the relations for $J_{Q_1,W_J}$ then imply that $K$ is actually a $J_{Q_1,W_J}$-submodule.  Since $V$ is simple, we must therefore have $K=V$.  Hence $y$ acts trivially on $V$, so that $V$ descends to a representation of the quotient
\[
J_{Q_1,W_J} / (y) \cong \C[x,z],
\]
which is commutative, and therefore only has one-dimensional simple modules.
\end{proof}
One-dimensional  modules are parametrized by the critical locus $\{df_J=0\}\subset\C^3$, where $f_J\colon \C^3\to\C$ is given by $f(x,y,z)=z y^2$. The motivic vanishing cycle of this function is
$\cL^2$, so the corresponding virtual motive is 
$(-\cL)^{-\frac{3}{2}}\cL^2 = -\cL^{\half}$, the numerator of~\eqref{eq_jordan}.

In the cases discussed in this section so far, the Jacobian algebra has only one-dimensional simple 
representations. Quantized affine three-space at roots of unity behaves differently (see \cite{BJL,DeLLeB}):

\begin{lemma} The algebra 
\[J_{Q_1,W_q} = \C\langle x, y, z\rangle / ( [x,y]_q, [y,z]_q, [z, x]_q )\]
corresponding to the potential $W=xyz-qxzy$ on $Q_1$ has simple modules of dimension $r > 1$ if and only if $q$ is a primitive $r$-th root of unity.  Moreover, the space of one-dimensional representations is independent of $q$ provided $q \ne 1$.
\label{lem_quantumC3}\end{lemma}

Note that the formula in \autoref{thm_quantumC3} has terms in precise
correspondence with the possible dimensions of simple representations.  Indeed, when $q$ is a primitive $r$th root of unity, we have
\begin{eqnarray*} U_{Q,W}(t) & =
&\displaystyle\Exp\left(\frac{2\cL-1}{\cL-1}\frac{t}{1-t} + (\cL-1)
  \frac{t^r}{1-t^r}\right) \nonumber \\
  & = &
\displaystyle\Exp\left(-\frac{-2\cL^{\oh}+\cL^{-\oh}}{\denom}\frac{t}{1-t}
- \frac{-\cL^{-\tfrac{3}{2}}\cdot \cL(\cL-1)^2}{\denom}\frac{t^r}{1-t^r}\right).
\label{eq_r2}
\end{eqnarray*}
The first summand  is again given by the virtual motive of one-dimensional representations, which are, of course, all simple. But there are also simple representations of dimension $r$, and they contribute the rational function $\frac{t^r}{1-t^r}$ to the generating series.  We may write its coefficient as
\begin{align}
\frac{-\cL^{-\tfrac{3}{2}}\cdot \cL(\cL-1)^2}{\denom} = \frac{[\A^1]_\vir[\cP^2\setminus Y]_{\vir}}{\denom},\label{eq_newcoeff}
\end{align}
where $Y \subset \cP^2$ denotes a triangle of projective lines.  To understand the origin of this triangle, we follow the noncommutative projective geometry approach of~\cite{ATVdB2,A,DeL}.  We observe that the moduli stack $\gM(J_{Q_1,W_q},r)$ of $r$-dimensional representations has a Zariski-open substack $\gM_{\rinv} \subset \gM(J_{Q_1,W_q},r)$ consisting of simple representations on which the central element $g=xyz$ acts invertibly.  Let us denote by $\A^3_q$ the noncommutative affine space defined by $J_{Q_1,W_q}$.  Then elements of $\gM_{\rinv}$ correspond to skyscraper sheaves on the noncommutative variety $\A^3_q \setminus \{g=0\}$, obtained by removing the coordinate planes.

According to~\cite{DeL}, the  stack $\gM_{\rinv}$ has a coarse moduli space $X$ that is a smooth threefold, and the stabilizers at each point are the scalars $\C^*$.  Moreover, the grading on $J_{Q_1,W_q}$ gives rise to a $\C^*$-action on representations, and this action makes $X$ into a principal $\C^*$-bundle over $\cP^2 \setminus Y$.  This fibration has the following interpretation: starting with a skyscraper sheaf $\sF$ on $\A^3_q \setminus \{g=0\}$, we take its direct image $\pi_*\sF$ along the quotient map $\pi : \A^3_q\setminus \{0\} \to \cP^2_q$ to the corresponding noncommutative $\cP^2$. The support of $\pi_*\sF$ defines a point in the ``centre'' of $\cP^2_q \setminus \{g=0\}$, which is the commutative variety $\P^2 \setminus Y$.

Assembling these facts, we can easily compute the virtual motive
\[
[\gM_\rinv]_\vir = \frac{[X]_\vir}{[\C^*]_\vir} = \frac{[\A^1\setminus\{0\}]_\vir[\cP^2\setminus Y]_\vir}{[\C^*]_\vir}
\]
Therefore $[\gM_\rinv]$ is very close to being the coefficient~\eqref{eq_newcoeff}, but the latter looks like a line bundle over $\cP^2\setminus Y$, rather than a $\C^*$-bundle.  We expect that this discrepancy can be explained by considering the closure $\overline{\gM_\rinv}\subset \gM(J_{Q_1,W_q},r)$, which evidently includes fixed points of the $\C^*$-action.

\subsubsection{The conjectures}

Notice that in all cases studied in the previous section, the answer had the general rational form
\begin{equation}
\label{conj_Q1}
U_{Q_1,W}(t) = \Exp\left(-\sum_{i=1}^{k} \frac{M_i}{\cL^\half-\cL^{-\half}} \frac{t^{m_i}}{1-t^{m_i}} \right),
\end{equation}
where $m_1=1, \ldots, m_k\in \N$ 
are the dimensions in which there exist {\em simple} modules for the
algebra $J_{Q_1, W}$, and $M_i\in\M$ are motivic expressions without
denominators, with $M_1$ being the virtual motive of the scheme parametrizing one-dimensional
simple modules. 

For the homogenized Weyl case, we have 
\[J_{Q_1,W_{hW}} = \C\langle x, y, z\rangle / ([x,y]-z^2, [x,z], [y,z]).
\]
\begin{lemma}\label{lem_weyl}
Every finite-dimensional simple module for $J_{Q_1,W_{hW}}$ has dimension one.
\end{lemma}
\begin{proof}
For $n \ge 0$, we have
\[
z^{n+2} = z^n[x,y] = [z^nx,y] - [z^n,y]x = [z^nx,y]\]
in $J_{Q_1,W_{hW}}$. 
Thus $z^{n+2}$ is a commutator for $n \ge 2$.  The rest of the proof follows that of Lemma~\ref{lem_jordan} above. 
\end{proof}

One-dimensional representations of $J_{Q_1,W_{hW}}$ are parametrized
by the (scheme-theoretic) critical locus of the function 
$f_{hW} = z^3$ on $\cC^3$, the double plane $\{z^2 = 0\}$. The motivic vanishing cycle
of $f_{hW}$ is $\cL^2(1-[\mu_3])$, and thus the corresponding virtual motive is $-\cL^\half(1-[\mu_3])$. 
Using \autoref{lem_weyl}, formula~\eqref{conj_Q1} turns into \autoref{conj_weyl} above. 

Let us finally turn to the Sklyanin algebras, with potential
\[
W_{a,b,c} = a\, xyz + b\, xzy + \frac{c}{3}(x^3 + y^3 + z^3).
\]
The only one-dimensional representation of these algebras is the trivial one, but the moduli space is highly non-reduced; it is the critical locus of the function
\[
f(x,y,z) = (a+b)xyz + \frac{c}{3}(x^3+y^3+z^3),
\] 
i.e.~it is the scheme-theoretic singular locus of a simple elliptic surface singularity of type $\widetilde E_6$, and therefore has length eight.

One-dimensional representations have the following motivic DT invariant:
\begin{lemma}\label{lem_Skly1d}
The virtual motive for the moduli space of one-dimensional representations is given by
\[
[\gM(J_{Q_1,W_{a,b,c}},1)]_\vir = \irL[3] \rbrac{[S_\DT,\mu_3] - [E_\DT](\cL-1) - 1}
\]
where, as before, $S_\DT$ is the affine triple cover of $\P^2\setminus E_\DT$ with its canonical action of $\mu_3$.
\end{lemma}  

\begin{proof}
The function $f$ has an isolated critical point at the origin, and $f^{-1}(0)$ is the cone over the elliptic curve $E_\DT \subset \P^2$.  Blowing up the origin in $\C^3$ gives a normal crossings resolution of $f$ on $X = {\rm Tot}(\OO_{\P^2}(-1))$, whose irreducible components are given by the zero section (with multiplicity three), and the total space of $\OO_{E_\DT}(-1)$. The result is then an immediate consequence of 
Denef and Loeser's formula~\cite[Thm.~3.3]{DL}.
\end{proof}

To motivate \autoref{conj_Skly}, we recall some basic facts about simple representations of the Sklyanin algebras~\cite{ATVdB2,A,DeLLeB,W}.  Let $r = |\sigma|$ be the order of the translation $\sigma : E_\pt \to E_\pt$.  Then higher-dimensional simple representations exist if and only if $r < \infty$.  In this case there, are explicit bounds on the dimensions~\cite{W}, which depend on whether of not $r$ is divisible by three.  Combining these bounds with \eqref{conj_Q1}, explains the rough shape of the expressions appearing in \autoref{conj_Skly}.

When $r$ is not divisible by three, somewhat more information is available.  In this case, all nontrivial simple representations have dimension $r$.  As in the discussion following \autoref{lem_Skly1d}, we can consider the moduli stack $\gM_{\rinv}$ of simple $r$-dimensional representations on which a certain cubic central element $g$ is invertible.  Its coarse moduli space is a $\C^*$-bundle over the complement of the elliptic curve $E_\pt/\sigma \subset \cP^2$.  The analogy between $E_\pt/\sigma$ and the triangle $Y \subset \cP^2$ in \eqref{eq_newcoeff} explains the appearance of the isogenous curve in part (ii) of the conjecture.

\subsubsection{Interpreting the results for the conifold quiver}

For the undeformed conifold, the formula of \autoref{thm_qconifold} includes~\cite[Thm.2.1]{MMNS}:
\begin{eqnarray*}
U_{Q,W}(t_1, t_2) & = &
\displaystyle\Exp\left(\frac{(\cL+\cL^2)t_1t_2-\cL^\oh(t_1+t_2)}{\cL-1}\sum_{n\ge0}(t_1t_2)^n\right)\\
& = &
\Exp\left(-\frac{-(\cL^{\frac32}+\cL^\oh)}{\denom}\frac{t_1t_2}{1-t_1t_2}\right.\\
&   & \ \ \ \ \ \ \ \ \left. -\frac{1}{\denom}\frac{t_1}{1-t_1t_2}
-\frac{1}{\denom}\frac{t_2}{1-t_1t_2}\right).
\end{eqnarray*}

In this case, there are two vertex simples, the representations with
dimension vectors $(1,0)$ and $(0,1)$; there are also simples of dimension vector
$(1,1)$. We are unable to give a systematic explanation of the denominators of these rational functions:
for all three dimension vectors, we get the denominator ${1-t_1t_2}$. 
The numerators are clear: they simply record the dimension vectors. 
As for the motivic coefficitents, for $(1,0)$ and $(0,1)$ the moduli 
space is a reduced point, in agreement with the coefficients above. On the other hand, 
it can be checked that a representation with dimension vector $(1,1)$ is simple if and only if
there is a nonzero arrow in each direction, and thus the parameter space of
simple representations is $(\C^2\setminus \pt)^2/\C^*$ which is the complement of
the zero section in the resolved conifold. Thus its virtual motive
is $-\cL^{-\frac32}(\cL+1)^2(\cL-1)$ which is close to, but not quite, the
numerator above. Instead, the numerator $-(\cL^{\frac32}+\cL^\oh)$ is the
virtual motive of the resolved (commutative) conifold $X$, the resolution 
$\pi\colon X\to Z$ of the singular conifold $Z = \mathop{\rm Spec} \C[x,y,z,t] / (xt-yz)$.

For the generic deformed conifold, we have
\[ U_{Q_\con, W_q}=\Exp\left(\frac{3\cL^{\frac{1}{2}}-\cL^{-\frac{1}{2}}}{\cL^{\frac{1}{2}}-\cL^{-\frac{1}{2}}}\frac{t_{0}t_{1}}{1-t_{0}t_{1}}-\frac{1}{\cL^{\frac{1}{2}}-\cL^{-\frac{1}{2}}}\frac{t_{0}+t_{1}}{1-t_{0}t_{1}}\right).
\]
Once again, the numerator of the first term is {\em not} the motive of the space of simples of dimension vector $(1,1)$. But it can be interpreted geometrically in the following way.  We begin by considering the global function $f = xt = yz$ on the conifold singularity $Z$, and its pullback $g=\pi^*(f)$ along the resolution $\pi : X \to Z$.  A straightforward calculation using the standard charts on the conifold shows that the virtual motive of $\crit(g)$ is the desired expression $-3\cL^{\frac{1}{2}}+\cL^{-\frac{1}{2}}$. 

Now, since $X$ is a smooth Calabi--Yau threefold, a choice of global holomorphic volume form gives rise to an isomorphism $\Omega^1_X \cong \wedge^2 T_X$ between forms and bivectors.  Under this isomorphism, the form $dg$ gives a Poisson structure $\pi \in H^0(X,\wedge^2T_X)$.   Via the derived equivalence $D(X) \cong D(J_{Q_\con,W_q})$ for $q=1$, the $q$-deformation of the Jacobi algebra corresponds to a noncommutative deformation of $X$ that we expect to be a deformation quantization of this Poisson structure.  In particular, the critical points of $g$ are exactly the zero-dimensional symplectic leaves of $\pi$, which are precisely the points one expects to quantize to skyscraper sheaves on the noncommutative deformation.

\subsubsection{Interpreting results for the cyclic quiver}

We will be brief, since the interpretations we can offer are analogous to the cases already studied. We refer for the details to~\cite{C}.  As usual, at $q=1$ the formula in Theorem~\ref{thm:cyclicAn} recovers the result of Bryan and Morrison~\cite{BM, Mor}
\[ U_{Q_{n+1}, W_1}(t)= \Exp\left(\frac{\cL^{\frac{3}{2}}+(n-1)\cL^{\frac{1}{2}}}{\cL^{\frac{1}{2}}-\cL^{-\frac{1}{2}}}\frac{y^{\delta}}{1-y^{\delta}}+\sum_{\alpha\in\Delta}\frac{\cL^\half}{\cL^{\frac{1}{2}}-\cL^{-\frac{1}{2}}}\frac{y^{\alpha}}{1-y^{\delta}}\right).
\]
The first coefficient here, up to sign, is the motive of the (unique) crepant resolution of the quotient 
singularity $\Spec Z$. The other coefficients are all virtual motives of the affine line. 
For $\alpha=\delta_i$, it is indeed
clear that the moduli space is simply the affine line, parameterized by the value of the loop arrow $b_i$. 
For generic $q$ on the other hand, the coefficient $(n+1)\cL^{\frac{1}{2}}-\cL^{-\frac{1}{2}}$ of the first term in 
Theorem~\ref{thm:cyclicAn} can again be interpreted, up to sign, as the virtual motive of the zero-set of a natural one-form on the crepant resolution. 

\section{Proofs}
\label{sec_proofs}

\subsection{Quivers with cuts}

A subset $I\subset E(Q)$ is called a cut of $(Q,W)$, if the potential
$W$ is homogeneous of degree $1$ for the arrow grading of $Q$ where
arrows in $I$ have degree 1 and arrows not in $I$ have degree zero.
Given a cut $I$ of $(Q,W)$, let $Q_I=(V(Q),E(Q)\ms I)$, and
let $J_{W,I}$ be the quotient of $\cC Q_I$ by the ideal $$(\dd_I W)=(\dd W/\dd a,a\in I).$$
Then by~\cite[Prop.1.7]{MMNS}, 
\begin{equation}\label{eq:dimred}
U_{Q,W}(t)=\sum_{\al\in\cN^{V(Q)}}(-\cL^\oh)^{\chi(\al,\al)+2d_I(\al)}\frac{[R(J_{W,I},\al)]}{[\GG_\al]}t^\al,\end{equation}
where $d_I(\al)=\sum_{(a:i\to j)\in I}\al_i\al_j$ for any $\al\in\cZ^{V(Q)}$.

\subsection{Deformations of affine three-space}

\subsubsection{Quantum affine three-space}\label{sec_qC3_proof}

In this section, we prove \autoref{thm_quantumC3}.  The proof will make heavy use of the assumption $q\neq 0$.  As already observed, this is precisely the condition for the Jacobian algebra $J_{Q_1, W_q}$ to be 3-Calabi--Yau, although we will not use this fact directly.

We begin by applying the cut $I=\{z\}$ to $(Q_1,W_q)$, which reduces the problem to studying representations of the algebras
\[
\C_q[x,y] = \C\abrac{x,y}/(xy-qyx)
\]
for $q \in \C^\times$.  More precisely, using formula~\eqref{eq:dimred}, we have
\[ U_{Q_1,W_q}(t) = \sum_{n\geq 0} \frac{[R_q(n)]}{[\GL(n)]} t^n  \]
where the variety
\[
R_q(n) = \{ (A,B) \in \End(V) \times \End(V)  : AB=qBA \}
\]
is the set of pairs of $q$-commuting $(n\times n)$-matrices, and $V= \C^n$ is a fixed $n$-dimensional vector space.   We think of $\C_q[x,y]$ geometrically as functions on a quantum plane $\A^2_q$, so that finite-dimensional representations of $\C_q[x,y]$ correspond to torsion coherent sheaves on $\A^2_q$; compare~\cite{BM}.

For the calculation, it will be useful to consider the four subvarieties
\begin{align*}
R_q^{I,I}(n) &= \set{(A,B)\in R_q(n)}{A,B\textrm{ are invertible}} \\
R_q^{I,N}(n) &= \set{(A,B)\in R_q(n)}{A\textrm{ is invertible and }B\textrm{ is nilpotent}} \\
R_q^{N,I}(n) &= \set{(A,B)\in R_q(n)}{A\textrm{ is nilpotent and }B\textrm{ is invertible}} \\
R_q^{N,N}(n) &= \set{(A,B)\in R_q(n)}{A,B\textrm{ are nilpotent}}
\end{align*}
The geometric interpretation of these subvarieties is as follows.  The quantum plane $\A^2_q$ contains a privileged pair of commutative affine lines $L_x, L_y \subset \A^2_q$ corresponding to the two-sided ideals $(x),(y) \subset \C_q[x,y]$. These lines intersect in the point $0 \in \A^2_q$ corresponding to the ideal $(x,y)$.  This gives a stratification of $\A^2_q$, and points of the varieties $R_q^{-,-}$ above correspond to sheaves on $\A^2_q$ whose supports are completely contained in one of the four strata.

Define the generating series
\[
U_{q}^{I,I}(t) = \sum_{n \ge 0} \frac{[R_q^{I,I}(n)]}{[\GL(n)]}t^n,
\]
and define $U_q^{N,I},U_q^{N,N}$ and $U_q^{I,N}$ similarly.  Then we have the following basic
\begin{lemma}[cf. {\cite[Lemma 1]{FF}}]\label{nilinvlem} For all $q \in \C^\times$, there is a factorisation
\[
U_{Q_1,W_q} = U_{q}^{I,I} \cdot   U_{q}^{I,N} \cdot  U_{q}^{N,I} \cdot  U_{q}^{N,N}
\]
of the universal generating series.
\end{lemma}

\begin{proof} Suppose given a point $(A,B) \in R_q(n)$ corresponding to a representation of $\C_q[x,y]$ on the vector space  $V = \C^n$.  Then the kernel and image of $A^N$ for $N \gg 0$ decompose $V$ into direct summands on which $A$ acts nilpotently and invertibly, respectively.  Since $q^N \ne 0$, the relation $A^NB = q^NBA^N$ implies that $B$ preserves the kernel and image of $A^N$.  Decomposing further using the action of $B$, we obtain a canonical decomposition
\[
V = V_{I,I} \oplus V_{I,N} \oplus V_{N,I} \oplus V_{N,N}
\]
into subrepresentations, corresponding to sheaves whose supports lie on a single stratum in $\A^2_q$.  The result now follows easily.
\end{proof}

\begin{lemma}\label{nil} The series $U_{q}^{N,N}$, $U_{q}^{N,I}$ and $U_q^{I,N}$ are independent of $q \in \C^\times$, namely
\[
U_q^{N,N}(t) = \Exp\rbrac{\frac{1}{\LL-1}\frac{t}{1-t}}, \]
and
\[ U_q^{N,I}(t) = U_q^{I,N}(t)= \Exp\rbrac{\frac{t}{1-t}}\]
for all $q \in \C^\times$.
\end{lemma}

\begin{proof}
The formulae in question for the case $q=1$ are easily extracted from the results of~\cite{BM}, so it suffices to demonstrate that the series are independent of $q$.
Moreover, since the equations defining $R_q(n) \subset \End(V) \times \End(V)$ are symmetric in the two factors, it is enough to show that the series $U^{N,I}_q$ and  $U^{N,*}_q=U^{N,N}_q\cdot U^{N,I}_q$ are independent of $q$.  Notice that, by the argument in \autoref{nilinvlem}, the series $U^{N,*}_q$ is the universal series for the sequence of varieties
\[
R_q^{N,*}(n) = \set{ (A,B)\in C_q(n) }{A\textrm{ is nilpotent}}.
\]
So, we must show that the motivic classes $[R_q^{N,*}(n)]$ and $[R_q^{N,I}(n)]$ are independent of $q$.  To do so, we consider the maps from $R^{N,*}_q$ and $R^{N,I}_q$ to the nilpotent cone in $\End(V) \cong \gln{}(n,\C)$, given by projection on the first factor.  We will show that, when restricted to a fixed nilpotent orbit, these maps are Zariski-locally trivial fibrations, and that their fibres are independent of $q$.

To this end, choose a nilpotent orbit $\OO\subset \End(V)$ and a matrix $A_0 \in \OO$.  Since $A_0$ is nilpotent and $q \ne 0$, we may choose an element $P \in \GL(V)$ such that $PA_0P^{-1} = q^{-1}A_0$.  Let $H \subset \GL(V)$ be the stabilizer of $A_0$ under the conjugation action, and let  $\h \subset \End(V)$ be its Lie algebra.  Acting by conjugation on $A_0$, we have a principal bundle $\GL(V) \to \OO$ with structure group $H$, which is Zariski-locally trivial since $H$ is special.
Given a local section $s : U \to \GL(V)$ over a subvariety $U \subset \OO$, consider maps
\[U \times H \to R_q^{N,I}\] 
and \[U \times \h {} \to R_q^{N,*},\]
both defined by the formula
\[
(A,h) \mapsto (A,s(A)hPs(A)^{-1}).
\]
One easily checks that these maps give local trivializations of $R_q^{N,I}$ and $R_q^{N,*}$ over~$U$.  Since the fibres are independent of $q$, the lemma follows.
\end{proof}

In order to complete the proof of \autoref{thm_quantumC3}, it remains to compute the series $U_q^{I,I}$.  We remark that the varieties $R^{I,I}_q$ parametrize modules over the localized algebra $\C_q[x^{\pm1},y^{\pm1}]$, which we can think of as functions on the quantum torus $T_q$ obtained by removing the lines $L_x,L_y \subset \A^2_q$.  As is well known, the algebra $\C_q[x^{\pm 1},y^{\pm1}]$ is isomorphic to the skew group ring $\cZ * (\C[x,x^{-1}])$ associated with the $\cZ$-action on the variety $\C^*$, where the generator $1 \in \cZ$ acts by multiplication by $q$.  If we denote by $[\C^*/\cZ]_q$ the quotient stack, then finite-dimensional modules over $\C_q[x^{\pm 1},y^{\pm 1}]$ are equivalent to torsion coherent sheaves on $[\C^* / \cZ]_q$.  

When $q$ is not a root of unity, the orbits of the $\cZ$-action are infinite, and hence there can be no nontrivial equivariant sheaves of finite length.  We therefore have
\[
U^{I,I}_q(t) = 1
\]
if $q$ is not a root of unity.

On the other hand, if $q$ is a primitive $r$th root of unity, we have an isomorphism of stacks $[\C^* /\cZ]_q \cong [\C^*/\cZ]_1$ induced by the $r$th-power map $\C^* \to \C^*$ and the inclusion $\cZ \cong r \cZ\subset \cZ$ of the stabilizer of the action.  Thus, there is an equivalence between finite-length sheaves on $T_q$ and finite-length sheaves on the commutative torus $T_1 = \Spec \C[u^{\pm1},v^{\pm 1}]$.  Since pulling back along the $r$th power map multiplies the length of a sheaf on $\C^*$ by $r$, this equivalence takes $n$-dimensional representations of $\C[u^{\pm 1},v^{\pm 1}]$ to $rn$-dimensional representations of $\C_q[x^{\pm 1},y^{\pm 1}]$.  We therefore have
\[
U^{I,I}_q(t) = U^{I,I}_1(t^r) = U^{N,N}_1(t^r)^{(\LL-1)^2} = \Exp\rbrac{(\LL-1)\frac{t^r}{1-t^r}},
\]
where the last two identities use the power structure and the results of \cite{BM}.

\subsubsection{The Jordan plane}\label{sec_Jordan_proof}

We now prove \autoref{thm_jordan}. Applying a cut along $I=\{z\}$ again,
we reduce to representations of the algebra
\[
\C_J[x,y] = \cC\abrac{x,y}/(xy-yx-y^2),
\]
the ring of functions on the Jordan plane.  We define the representation varieties
\[
R_J(n) = \set{(A,B) \in \End(V)\times \End(V)}{[A,B]=B^2}
\]
for $n \ge 0$, where once again $V$ denotes a fixed $n$-dimensional vector space.  Using \eqref{eq:dimred}, we have the equality
\[
U_{Q_1,W_J}(t) = \sum_{n \ge 0} \frac{[R_J(n)]}{[\GL(n)]} t^n.
\]
The series on the right can be easily computed using the results in the previous section.

Indeed, if $(A,B) \in R_J(n)$ then $B$ is nilpotent, as observed in \autoref{lem_jordan}; see also \cite[Lemma 2.1]{I}.  Projection on the second factor therefore gives a map from $R_J(n)$ to the nilpotent cone in $\End(V)$.  Over a fixed nilpotent orbit $\OO \subset \gln{n}$ this map is an affine bundle for the vector bundle over $\OO$ whose fibre at $B \in \OO$ is the centralizer of $B$ in $\End(V)$.  Hence $R_J(n)$ has the same motivic class as the variety of pairs of commuting matrices, the second of which is nilpotent.  But these varieties are precisely those considered in the proof of \autoref{nil} in the case $q=1$.  We therefore  conclude that 
\[
U_{Q_1,W_J}(t) = U^{N,I}_q(t) \cdot U^{N,N}_q(t) = \Exp\rbrac{\frac{\LL}{\LL-1}\frac{t}{1-t}},
\]
proving \autoref{thm_jordan}.

\subsection{The deformed conifold}

In this section, we sketch the proof of \autoref{thm_qconifold}, 
which follows that of \autoref{thm_quantumC3}
and \cite[Sect. 2.2]{MMNS}.  We refer the reader to \cite{C} for full details.

 Using $I=\{a_1\}$ as the cut, we are lead 
to considering representations of the quiver in \autoref{fig_Q2_cut}, with the single relation $b_1a b_2=qb_2ab_1$, where $a = a_2$.
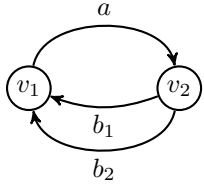
\begin{figure}[h]
\begin{center}
\begin{tikzpicture}[>=stealth',auto,node distance=2cm,
  thick,main node/.style={circle,draw,inner sep=2pt,minimum size=0.2cm}]
\node[main node] (1) {$v_1$};
\node[main node] (2) [right of=1] {$v_2$};
\draw [->] (1) .. controls (0.2,1) and (1.8,1) .. (2);
\draw [->] (2) .. controls (1.8,-1) and (0.2,-1) .. (1);
\draw [->] (2) .. controls (1.2,-0.3) and (0.8,-0.3) .. (1);
\draw (1,-0.5) node {$b_1$};
\draw (1,1.05) node {$a$};
\draw (1,-1.08) node {$b_2$};
\end{tikzpicture}
\caption{The cut of $Q_2$ along $a_1$}\label{fig_Q2_cut}
\end{center}
\end{figure}

Thus, to compute the generating series, we must consider the varieties
\begin{equation*}
\{(A,B_1,B_2)\in \Hom(V_1, V_2)\times\Hom(V_2, V_1)^{\times 2} \colon B_{1}AB_{2}-qB_{2}AB_{1}=0\},
\end{equation*}
where $d=(d_1, d_2)$ is a dimension vector for the quiver $Q_2$, and $V_i$ are fixed vector spaces
of dimension $d_i$.  Given an element $(A,B_{1},B_{2})$ of such a variety, consider the linear map 
\begin{equation*}
 A_{2}\oplus B_{2}\in \Hom(V_{0},V_{1})\oplus \Hom(V_{1},V_{0})\subset \End(V),
\end{equation*}
where $V = V_{0}\oplus V_{1}$.  As in \autoref{nilinvlem}, the relation implies that we can decompose $V = V^N\oplus V^I$ into subrepresentations on which $A \oplus B_2$ acts nilpotently and invertibly, respectively, so that the generating series factors 
\[
U_{Q_2,W_q} = U^N\cdot U^I
\]
into nilpotent and invertible contributions.

Once again, one shows that the series $U^N$ is independent of $q$, while the computation of the series $U^I$ can be reduced to the study of the $q$-commuting varieties of \autoref{sec_qC3_proof}.  Combining that calculation with the formulae in \cite[Sect. 2.2]{MMNS} for the undeformed conifold yields the result.

\subsection{The cyclic quiver} The proof of Theorem~\ref{thm:cyclicAn} proceeds analogously to the conifold case, using dimensional reduction and appropriate splittings, reducing the calculation to the case $q=1$ already done in~\cite{BM, Mor}. Once again, we refer for the details to~\cite{C}.

\bibliographystyle{hyperamsplain}
\bibliography{NCDT}

\providecommand{\bysame}{\leavevmode\hbox to3em{\hrulefill}\thinspace}
\providecommand{\MR}{\relax\ifhmode\unskip\space\fi MR }
\providecommand{\MRhref}[2]{%
  \href{http://www.ams.org/mathscinet-getitem?mr=#1}{#2}
}
\providecommand{\href}[2]{#2}
\begin{thebibliography}{10}

\bibitem{ATVdB}
M.~Artin, J.~Tate, and M.~Van~den Bergh, \emph{Some algebras associated to
  automorphisms of elliptic curves}, The {G}rothendieck {F}estschrift, {V}ol.\
  {I}, Progr. Math., vol.~86, Birkh\"auser Boston, Boston, MA, 1990,
  pp.~33--85.

\bibitem{ATVdB2}
\bysame, \emph{Modules over regular algebras of dimension {$3$}},
  \href{http://dx.doi.org/10.1007/BF01243916}{Invent. Math. \textbf{106}
  (1991)}, no.~2, 335--388.

\bibitem{A}
M.~Artin, \href{http://dx.doi.org/10.1090/conm/124/1144023}{\emph{Geometry of
  quantum planes}}, Azumaya algebras, actions, and modules ({B}loomington,
  {IN}, 1990), Contemp. Math., vol. 124, Amer. Math. Soc., Providence, RI,
  1992, pp.~1--15.

\bibitem{AS}
M.~Artin and W.~F. Schelter, \emph{Graded algebras of global dimension {$3$}},
  \href{http://dx.doi.org/10.1016/0001-8708(87)90034-X}{Adv. in Math.
  \textbf{66} (1987)}, no.~2, 171--216.

\bibitem{Be}
K.~Behrend, \emph{Donaldson-{T}homas type invariants via microlocal geometry},
  \href{http://dx.doi.org/10.4007/annals.2009.170.1307}{Ann. of Math. (2)
  \textbf{170} (2009)}, no.~3, 1307--1338.

\bibitem{BBS}
K.~Behrend, J.~Bryan, and B.~Szendr{\H{o}}i, \emph{Motivic degree zero
  {D}onaldson-{T}homas invariants},
  \href{http://dx.doi.org/10.1007/s00222-012-0408-1}{Invent. Math. \textbf{192}
  (2013)}, no.~1, 111--160.

\bibitem{BJL}
D.~Berenstein, V.~Jejjala, and R.~G. Leigh, \emph{Marginal and relevant
  deformations of {$N=4$} field theories and non-commutative moduli spaces of
  vacua}, \href{http://dx.doi.org/10.1016/S0550-3213(00)00394-1}{Nuclear Phys.
  B \textbf{589} (2000)}, no.~1-2, 196--248.

\bibitem{Bo}
L.~Borisov, \emph{Class of the affine line is a zero divisor in the
  Grothendieck ring}, \href{http://arxiv.org/abs/1412.6194}{{\tt 1412.6194}}.

\bibitem{BM}
J.~Bryan and A.~Morrison, \emph{Motivic classes of commuting varieties via
  power structures},
  \href{http://dx.doi.org/10.1090/S1056-3911-2014-00657-3}{J. Algebraic Geom.
  \textbf{24} (2015)}, no.~1, 183--199.

\bibitem{C}
A.~Cazzaniga, \emph{On Some Computations of Refined {D}onaldson-{T}homas
  Invariants}, DPhil Thesis, University of Oxford, 2015.

\bibitem{DM2}
B.~Davison and S.~Meinhardt, \emph{Motivic {DT}-invariants of (-2) curves},
  \href{http://arxiv.org/abs/1208.2462}{{\tt 1208.2462}}.

\bibitem{DM}
\bysame, \emph{Motivic {DT}-invariants for the one loop quiver with potential},
  Geom. Topol. (to appear), \href{http://arxiv.org/abs/1108.5956}{{\tt
  1108.5956}}.

\bibitem{DeL}
K.~De~Laet, \emph{Geometry of representations of quantum spaces},
  \href{http://arxiv.org/abs/1405.1938}{{\tt 1405.1938}}.

\bibitem{DeLLeB}
K.~De~Laet and L.~Le~Bruyn, \emph{The geometry of representations of
  3-dimensional {S}klyanin algebras},
  \href{http://dx.doi.org/10.1007/s10468-014-9515-6}{Algebr. Represent. Theory
  \textbf{18} (2015)}, no.~3, 761--776.

\bibitem{DL}
J.~Denef and F.~Loeser, \emph{Geometry on arc spaces of algebraic varieties},
  European {C}ongress of {M}athematics, {V}ol. {I} ({B}arcelona, 2000), Progr.
  Math., vol. 201, Birkh\"auser, Basel, 2001, pp.~327--348.

\bibitem{DG}
T.~Dimofte and S.~Gukov, \emph{Refined, motivic, and quantum},
  \href{http://dx.doi.org/10.1007/s11005-009-0357-9}{Lett. Math. Phys.
  \textbf{91} (2010)}, no.~1, 1--27.

\bibitem{DT}
S.~K. Donaldson and R.~P. Thomas, \emph{Gauge theory in higher dimensions}, The
  geometric universe ({O}xford, 1996), Oxford Univ. Press, Oxford, 1998,
  pp.~31--47.

\bibitem{FF}
W.~Feit and N.~J. Fine, \emph{Pairs of commuting matrices over a finite field},
  Duke Math. J \textbf{27} (1960), 91--94.

\bibitem{GSLMH}
S.~M. Gusein-Zade, I.~Luengo, and A.~Melle-Hern{\'a}ndez, \emph{A power
  structure over the {G}rothendieck ring of varieties},
  \href{http://dx.doi.org/10.4310/MRL.2004.v11.n1.a6}{Math. Res. Lett.
  \textbf{11} (2004)}, no.~1, 49--57.

\bibitem{I}
N.~K. Iyudu, \emph{Representation spaces of the {J}ordan plane},
  \href{http://dx.doi.org/10.1080/00927872.2013.788184}{Comm. Algebra
  \textbf{42} (2014)}, no.~8, 3507--3540.

\bibitem{JS}
D.~Joyce and Y.~Song, \emph{A theory of generalized {D}onaldson-{T}homas
  invariants}, \href{http://dx.doi.org/10.1090/S0065-9266-2011-00630-1}{Mem.
  Amer. Math. Soc. \textbf{217} (2012)}, no.~1020, iv+199.

\bibitem{KS}
M.~Kontsevich and Y.~Soibelman, \emph{Stability structures, motivic
  {D}onaldson--{T}homas invariants and cluster transformations},
  \href{http://arxiv.org/abs/0811.2435}{{\tt 0811.2435}}.

\bibitem{L}
E.~Looijenga, \emph{Motivic measures}, Ast\'erisque (2002), no.~276, 267--297.
  S{\'e}minaire Bourbaki, Vol. 1999/2000.

\bibitem{Mor}
A.~Morrison, \emph{Motivic invariants of quivers via dimensional reduction},
  \href{http://dx.doi.org/10.1007/s00029-011-0081-z}{Selecta Math. (N.S.)
  \textbf{18} (2012)}, no.~4, 779--797.

\bibitem{MMNS}
A.~Morrison, S.~Mozgovoy, K.~Nagao, and B.~Szendr{\H{o}}i, \emph{Motivic
  {D}onaldson-{T}homas invariants of the conifold and the refined topological
  vertex}, \href{http://dx.doi.org/10.1016/j.aim.2012.03.030}{Adv. Math.
  \textbf{230} (2012)}, no.~4-6, 2065--2093.

\bibitem{Moz}
S.~Mozgovoy, \emph{Wall-crossing formulas for framed objects},
  \href{http://dx.doi.org/10.1093/qmath/has010}{Q. J. Math. \textbf{64}
  (2013)}, no.~2, 489--513.

\bibitem{N}
K.~Nagao, \emph{{W}all-crossing of the motivic {D}onaldson--{T}homas
  invariants}, \href{http://arxiv.org/abs/1103.2922}{{\tt 1103.2922}}.

\bibitem{OU}
S.~Okawa and K.~Ueda, \emph{Noncommutative quadric surfaces and noncommutative
  conifolds}, \href{http://arxiv.org/abs/1403.0713}{{\tt 1403.0713}}.

\bibitem{Sze}
B.~Szendr{\H{o}}i, \emph{Non-commutative {D}onaldson-{T}homas invariants and
  the conifold}, \href{http://dx.doi.org/10.2140/gt.2008.12.1171}{Geom. Topol.
  \textbf{12} (2008)}, no.~2, 1171--1202.

\bibitem{T}
R.~P. Thomas, \emph{A holomorphic {C}asson invariant for {C}alabi-{Y}au
  3-folds, and bundles on {$K3$} fibrations}, J. Differential Geom. \textbf{54}
  (2000), no.~2, 367--438.

\bibitem{VdB}
M.~van~den Bergh, \emph{Non-commutative crepant resolutions}, The legacy of
  {N}iels {H}enrik {A}bel, Springer, Berlin, 2004, pp.~749--770.

\bibitem{W}
C.~Walton, \emph{Representation theory of three-dimensional {S}klyanin
  algebras}, \href{http://dx.doi.org/10.1016/j.nuclphysb.2012.02.015}{Nuclear
  Phys. B \textbf{860} (2012)}, no.~1, 167--185.

\end{thebibliography}

\end{document}